\documentclass[A4]{amsart}
\usepackage[square,sort,comma,numbers]{natbib}
\usepackage{amsmath, amssymb, amstext, amsfonts, textcomp, amsxtra, amsbsy, amsgen, amsopn, amscd, mathrsfs, amsthm, latexsym, array}
\usepackage[textwidth=16cm,textheight=22cm,centering]{geometry}

\usepackage{color}
\usepackage{graphicx}

\usepackage[all]{xy}

\usepackage{bbm, dsfont}

\usepackage{mathtools}
\mathtoolsset{showonlyrefs}

\newtheorem{theorem}             {Theorem}  [section]

\newtheorem{lemma}      [theorem]{Lemma}
\newtheorem{corollary}  [theorem]{Corollary}
\newtheorem{proposition}[theorem]{Proposition}

\numberwithin{equation}{section} \everymath{\displaystyle}

\newcommand{\Sch}{\mathcal{S}}
\newcommand{\sgn}{{\rm sgn}}

\newcommand{\ep}{\varepsilon}

\newcommand{\gp}[1]{\mathbf{#1}}
\newcommand{\GL}{{\rm GL}}

\newcommand{\SL}{{\rm SL}}

\newcommand{\Z}{\mathbb{Z}}

\newcommand{\Mat}{{\rm M}}

\newcommand{\lcm}{{\rm lcm}}

\newcommand{\legendre}[2]{\left( \tfrac{#1}{#2} \right)}

\newcommand{\Q}{\mathbb{Q}}

\newcommand{\C}{\mathbb{C}}
\newcommand{\F}{\mathbf{F}}
\newcommand{\A}{\mathbb{A}}

\newcommand{\vp}{\mathfrak{p}}

\newcommand{\norm}[1][\cdot]{\lvert #1 \rvert}
\newcommand{\extnorm}[1]{\left\lvert #1 \right\rvert}
\newcommand{\Norm}[1][\cdot]{\lVert #1 \rVert}

\newcommand{\rpL}{{\rm L}}
\newcommand{\rpR}{{\rm R}}
\newcommand{\Bas}{\mathcal{B}}

\newcommand{\Res}{{\rm Res}}

\newcommand{\cond}{\mathfrak{c}}

\newcommand{\Zeta}{\mathrm{Z}}

\newcommand{\Vol}{{\rm Vol}}
\makeatletter

\newcommand{\Rmnum}[1]{\expandafter\@slowromancap\romannumeral #1@}
\makeatother

\newcommand{\id}{\mathbbm{1}}

\newcommand{\psmatrix}[4]{\left(\begin{smallmatrix}#1&#2\\#3&#4\end{smallmatrix}\right)}

\title{Hybrid subconvexity and the partition function}
\author{Nickolas Andersen and Han Wu}

\begin{document}

\begin{abstract}
	We give an upper bound for the error term in the Hardy-Ramanujan-Rademacher formula for the partition function.
	The main input is a new hybrid subconvexity bound for the central value $L(\tfrac 12,f\times (\tfrac{q}{\cdot}))$ in the $q$ and spectral parameter aspects, where $f$ is a Hecke-Maass cusp form for $\Gamma_0(N)$ and $q$ is a fundamental discriminant.
\end{abstract}

	\maketitle

\section{Introduction}

Of the many fruits of the famous collaboration of Hardy and Ramanujan, the invention of the circle method is  among the most lasting and influential.
Their first application of the circle method  was the discovery in 1918 of an asymptotic formula for the partition function $p(n)$ which is precise enough to compute the exact value of $p(n)$ when $n$ is sufficiently large.
The Hardy-Ramanujan formula is, in the notation of \cite{HR},
\begin{equation} \label{eq:HR-formula}
	p(n) = \sum_{q=1}^\nu A_q(n)\phi_q(n) + O_\alpha(n^{-1/4}), \qquad \nu = \lfloor \alpha\sqrt{n} \rfloor,
\end{equation}
where $\alpha$ is any fixed positive real number, $A_q(n)$ is a finite sum of $24q$-th roots of unity, and
\begin{equation}
	\phi_q(n) = \frac{\sqrt q}{2\pi\sqrt{2}} \frac{d}{dn} \left( \frac{\exp(C\lambda_n/q)}{\lambda_n} \right), \qquad C=\pi\sqrt{\tfrac 23}, \quad \lambda_n=\sqrt{n-\tfrac 1{24}}.
\end{equation}
Since $p(n)$ is an integer and the error term tends to zero as $n\to\infty$, the exact value of $p(n)$ is the nearest integer to the sum in \eqref{eq:HR-formula} for all sufficiently large $n$.
However, the series obtained by replacing $\nu$ by $\infty$ in \eqref{eq:HR-formula} diverges
because $\phi_q(n)\gg_n \sqrt q$ and for each $n$ there are infinitely many $q$ for which $A_q(n)\gg \sqrt q$ (see \cite{lehmer1} for the latter fact).

In 1936, by carefully refining the contour used in the Hardy-Ramanujan circle method, Rademacher \cite{rademacher} showed that a slight modification of the sum in \eqref{eq:HR-formula} leads to an absolutely convergent infinite series whose value equals $p(n)$.
In fact, the only modification required is replacing the exponential function in the definition of $\phi_q(n)$ by the hyperbolic sine function, that is,
\begin{equation} \label{eq:HRR-formula}
	p(n) = \sum_{q=1}^\infty A_q(n)\tilde\phi_q(n), \qquad \tilde\phi_q(n) = \frac{\sqrt q}{\pi\sqrt{2}} \frac{d}{dn} \left( \frac{\sinh(C\lambda_n/q)}{\lambda_n} \right).
\end{equation}

Suppose that we truncate Rademacher's series at $q=\nu$, as in \eqref{eq:HR-formula}, and write
\begin{equation}
	p(n) = \sum_{q=1}^\nu A_q(n)\tilde\phi_q(n) + R(n,\nu).
\end{equation}
How fast does $R(n,\nu)$ decay as $n\to\infty$?
Since $\tilde\phi_q(n)$ is exponentially large when $q\ll \sqrt n$, it makes sense to set $\nu=\lfloor \alpha \sqrt{n} \rfloor$ for some $\alpha>0$, as in \eqref{eq:HR-formula}.
Rademacher showed in \cite{rademacher} that $R(n,\alpha\sqrt n)\ll_\alpha n^{-1/4}$, which matches the error term in \eqref{eq:HR-formula}.
It is apparent that any improvement to this bound requires a careful study of the $A_q(n)$, which are generalized Kloosterman sums given by
\begin{equation}
	A_q(n) = \sqrt{-i}\sum_{\substack{d\bmod q \\ (d,q)=1}} \bar\nu_\eta\left(\!\begin{pmatrix} a & \ast \\ q & d \end{pmatrix}\!\right) e\left(\frac{a+d}{24q} - \frac{nd}{q}\right),
\end{equation}
where $ad\equiv 1\pmod{q}$, $e(x)=e^{2\pi i x}$, and $\nu_\eta:\SL_2(\Z)\to\C$ is the multiplier system for the Dedekind eta function given by
\begin{equation}
    \nu_\eta\left(\!\begin{pmatrix} a & b \\ c & d \end{pmatrix}\!\right) = \frac{\eta(\frac{az+b}{cz+d})}{\sqrt{cz+d}\,\eta(z)}, \qquad \eta(z) = e(\tfrac{1}{24}z)\prod_{k=1}^\infty (1-e(kz)).
\end{equation}

A priori, $|A_q(n)|\leq q$, which is enough to show that the series \eqref{eq:HRR-formula} converges absolutely since we have $\tilde\phi_q(n)\ll q^{-5/2}$ for $q\gg \sqrt n$.
In 1937, Lehmer \cite{lehmer1} showed that 
\begin{equation} \label{eq:lehmer-weil-bound}
|A_q(n)| < 2^{\omega_o(q)}\sqrt q, 
\end{equation}
where $\omega_o(q)$ is the number of distinct odd primes dividing $q$.
He concluded \cite{lehmer2} that 
\[
    R(n,\alpha\sqrt n)\ll_{\alpha} n^{-1/2}\log n.
\]
This upper bound is essentially best possible if one estimates the sum of the absolute values of the terms in $R(n,\nu)$.
Any improvement must take into account the sign changes of the $A_q(n)$.
Lehmer's record held until 2010, when Folsom and Masri \cite{FM} broke the $-1/2$ exponent barrier when $24n-1$ is squarefree.
This exponent was improved in 2015 by the first author and Ahlgren \cite{AA}, then again in 2018 by Ahlgren and Dunn \cite{AD}, who proved that for all $n$ and for all $\ep>0$ we have
\begin{equation}
	R(n,\alpha\sqrt n) \ll_{\alpha,\ep} n^{-1/2-1/147+\ep}. 
\end{equation}
The main result of this paper is an improvement on this upper bound.

\begin{theorem}\label{thm:main} \label{thm:partition}
	Write $24n-1=tw^2$, where $t$ is squarefree.
	Then
	\begin{equation}
		R(n,\alpha\sqrt n) \ll_{\alpha,\ep} t^{-1/36}w^{-1/6}n^{-1/2+\ep}.
	\end{equation}
\end{theorem}

The upper bound in Theorem~\ref{thm:main} follows from an estimate (see \eqref{eq:sum-acn-main} below) for the weighted sums of Kloosterman sums
\begin{equation} \label{eq:Aqn-sum}
    \sum_{q\leq x} \frac{A_q(n)}{q}.
\end{equation}
Lehmer's inequality \eqref{eq:lehmer-weil-bound} yields the trivial upper bound $\sqrt x\log x$ for the sum \eqref{eq:Aqn-sum}.
We will give an improved bound in the critical range $x\asymp \sqrt n$.
In Section~\ref{sec:kloo-coeffs} we relate, via a Kuznetsov trace formula, the sum \eqref{eq:Aqn-sum} to a sum of coefficients of Maass cusp forms $u_j$ of half-integral weight whose transformation law involves the eta multiplier $\nu_\eta$.
This closely follows the arguments in \cite{AA} and \cite{AD}.
We depart from the arguments in \cite{AA} and \cite{AD}
at Proposition~\ref{prop:sum-walds} below, which
is a Waldspurger-type formula relating the cusp form coefficient $\mu_j(d)$ to the central value $L(\tfrac 12,\varphi\times\chi_d)$, where $\varphi$ is the Shimura lift of $u_j$ associated to the discriminant $d$ (see Section~\ref{sec:shimura} for notation and details).
Then in Section~\ref{sec:subconvexity}, using a version of Motohashi's formula proved by the second author in \cite{Wu11}, we prove the hybrid subconvexity bound
\begin{equation} \label{eq:sub-intro} L \left( \tfrac{1}{2}, \varphi \times \chi_d \right) \ll_{\ep} (|r| \norm[d] N)^{\ep} \left( |r| \lcm(\norm[d],N) \right)^{\frac{1}{3}}, \end{equation}
where $N$ is the level of $\varphi$ and $r$ is the spectral parameter of $\varphi$.
This generalizes a result of Young \cite{Yo17} to arbitrary level $N$.
(Note that \eqref{eq:sub-intro} is not subconvex in the $N$ aspect, but it is in the $d$ and $r$ aspects.)
This leads to our bound \eqref{eq:sum-acn-main} 
from which Theorem~\ref{thm:main} follows quickly.

\section{Sums of Kloosterman sums}

\label{sec:kloo-coeffs}

In this section we outline the proof of Theorem~\ref{thm:main}, postponing many of the details to later sections.
In what follows, we will use the same setup and notation as \cite{AA}, with the notable exception that we are normalizing the Maass form coefficients differently (see \eqref{eq:uj-normalize} below); for more details, consult that paper, especially Section~2.
Through the rest of the paper, we will encounter the Bessel functions $I_\nu$, $J_\nu$, and $K_\nu$, and the Whittaker functions $M_{\mu,\nu}$ and $W_{\mu,\nu}$. 
Definitions and properties of these functions can be found in Sections~10  and 13 of \cite{dlmf}.

Let $\mathcal S_{k}(N,\nu,r)$ denote the vector space of Maass cusp forms of weight $k$ on $\Gamma_0(N)$ with multiplier system $\nu$ and spectral parameter $r$.
Each function $u\in \mathcal S_{k}(N,\nu,r)$ satisfies
\begin{equation}
    u\left(\frac{az+b}{cz+d}\right) = \nu(\gamma) \left(\frac{cz+d}{|cz+d|}\right)^{k} u(z) \qquad \text{ for all }\gamma = \left(\begin{matrix}a&b\\c&d\end{matrix}\right) \in \Gamma_0(N)\\
\end{equation}
and
\begin{equation}
    -\Delta_{k} u := \left(-y^2\left(\partial_x^2+\partial_y^2\right) + ik y \partial_x\right) u = \left(\tfrac 14+r^2\right)u
\end{equation}
and vanishes at the cusps of $\Gamma_0(N)$.
We will be primarily interested in the space $\mathcal S_{1/2}(1,\nu_\eta,r)$.
It is shown in Section~5 of \cite{AA} that $\mathcal S_{1/2}(1,\nu_\eta,r)=\{0\}$ unless $r=i/4$ or $r> 0$.
Let $\mathcal S_{1/2}^\ast(1,\nu_\eta)$ denote the span of the union of all of the spaces $\mathcal S_{1/2}(1,\nu_\eta,r)$ for $r> 0$ and let $\{u_j\}_{j\geq 1}$ denote an orthonormal basis for $\mathcal S_{1/2}^\ast(1,\nu_\eta)$ with $r_j$ the spectral parameter of $u_j$.
For convenience in working with the Hecke operators, we will normalize the Fourier coefficients of $u_j$ by
\begin{equation} \label{eq:uj-normalize}
    u_j(z) = \sum_{n\equiv 1(24)} \mu_j(n) W_{\frac 14 \sgn(n),ir_j}(\tfrac \pi6 |n|y)e(\tfrac 1{24} n x).
\end{equation}
These coefficients are related to the $\rho_j(n)$ in \cite{AA} by $\rho_j(n) = \mu_j(24n-23)$.

Theorem~4.1 of \cite{AA} is a Kuznetsov trace formula for the sums $A_c(n)$ since $A_c(n)=\sqrt{-i}\, S(1,1-n,c,\chi)$ in the notation of that paper.
To state the formula, we first let $\phi:[0,\infty)\to \C$ be a four times continuously differentiable function satisfying
\begin{equation}
    \phi(0)=\phi'(0)=0, \quad \phi^{(j)}(x) \ll x^{-2-\ep} \quad (j=0,\ldots,4) \text{ as }x\to\infty,
\end{equation}
and let $\check\phi$ denote the integral transform
\begin{equation}
    \check\phi(r) = \cosh(\pi r) \int_0^\infty K_{2ir}(u) \phi(u) u^{-1} \, du.
\end{equation}
Then for $n\geq 1$ we have
\begin{equation}
\sum_{c=1}^\infty \frac{A_c(n)}{c} \phi \left( \frac{\pi\sqrt{24n-1}}{6c} \right)
= 
\tfrac 13 \sqrt{24n-1} \sum_{r_j>0} \frac{\overline{\mu}_j(1)\mu_j(1-24n)}{\cosh \pi r_j} \check \phi(r_j). 
\end{equation}
We choose $\phi = \phi_{a,x,T}$ as in Section~6 of \cite{AA}, where $a=\tfrac \pi 6 \sqrt{24n-1}$ and $T=x^{1/3}$.
Breaking the spectral sum into three ranges as in the proof of \cite[Proposition 9.2]{AA},
we find that
\begin{equation}
\sum_{x\leq c\leq 2x} \frac{A_c(n)}{c}
\ll
x^{1/6}\log x + \mathcal M_1 + \mathcal M_2 + \sum_{\ell=0}^\infty \mathcal M_3\left(2^\ell \frac{a}{x}\right),
\end{equation}
where
\begin{align}
    \mathcal M_1 &= \sqrt{24n-1}\sum_{0<r_j<a/8x} \frac{|\mu_j(1)\mu_j(1-24n)|}{\cosh(\pi r_j)} r_j^{-3/2}e^{-r_j/2}, \\
    \mathcal M_2 &= x\sum_{a/8x\leq r_j<a/x} \frac{|\mu_j(1)\mu_j(1-24n)|}{\cosh(\pi r_j)}, \\
    \mathcal M_3(A) &= \frac{\sqrt{24n-1}}{A^{3/2}}
    \min\left(1,\frac{x^{1/3}}{A}\right)
    \sum_{A \leq r_j < 2A} \frac{|\mu_j(1)\mu_j(1-24n)|}{\cosh(\pi r_j)}.
\end{align}
After applying the Cauchy-Schwarz inequality to each sum above, we see that we need an upper bound for the sums
\begin{equation}\label{eq:two-sums}
    \sum_{r_j\leq x}\frac{|\mu_j(1)|^2}{\cosh(\pi r_j)}
    \quad\text{ and }\quad
    (24n-1) \sum_{r_j\leq x}\frac{|\mu_j(1-24n)|^2}{\cosh(\pi r_j)}
\end{equation}
for all $n\geq 1$.
One estimate is provided by Theorem~4.1 of \cite{ADuke}:
for all $m\neq 0$ we have
\begin{equation} \label{eq:mean-val-estimate}
    |m|\sum_{r_j\leq x}\frac{|\mu_j(24m-23)|^2}{\cosh(\pi r_j)} \ll x^{-\sgn(m)/2}\left(x^2+|m|^{1/2}\right)(|m|x)^\ep.
\end{equation}
This result uses \eqref{eq:lehmer-weil-bound} and the fact that the negatively indexed Fourier coefficients in weight $1/2$ are positively indexed Fourier coefficients in weight $-1/2$ (see (2.17) of \cite{AA}).
When $x$ is large compared with $|m|$, say $x\gg |m|^{1/4}$ the estimate \eqref{eq:mean-val-estimate} is essentially sharp,
so it suffices as an estimate for the first sum in \eqref{eq:two-sums}.
We require an estimate in the complementary range $x\ll |m|^{1/4}$ for the second sum.
It will depend on the factorization of $1-24n$; for the remainder of the section, write
\begin{equation} \label{eq:24n-dw}
    1-24n=dw^2, \quad \text{ where $d\equiv 1 \!\!\!\! \pmod{24}$ is a negative fundamental discriminant.}
\end{equation}

There are Hecke operators $T_{p^2}$ on the spaces $\mathcal S_{1/2}(1,\nu_\eta,r)$ for primes $p\geq 5$ (see Section~2.6 of \cite{AA}); they act on Fourier expansions via
\begin{equation}
    T_{p^2} u_j(z) = \sum_{n\equiv 1(24)} \left( p \mu_j(p^2n) + p^{-1/2} \legendre{12n}{p} \mu_j(n) + p^{-1} \mu_j(n/p^2) \right)W_{\frac 14\sgn(n),ir_j}\left(\tfrac \pi6 | n|y\right) e(\tfrac{1}{24} nx),
\end{equation}
where $\mu_j(n/p^2)=0$ if $p^2\nmid n$.
We can choose our basis $\{u_j\}$ such that each $u_j$ is a an eigenform of $T_{p^2}$ for all $p\geq 5$, i.e.
\begin{equation} \label{eq:eigenform-def}
    T_{p^2} u_j = \lambda_j(p) u_j.
\end{equation}
If $\mu_j(d) = 0$ then \eqref{eq:eigenform-def} implies that $\mu_j(dw^2)=0$ for all $w$.
Since $u_j\neq 0$ there exists a squarefree $d$ such that $\mu_j(d)\neq 0$; in what follows we will assume that $d$ is chosen with this property.

Our next aim is to construct the Shimura lift $S_d(u_j)$ associated to the fundamental discriminant $d$; for $d>0$ this is done in Section~5 of \cite{AA}, but here we need it for $d<0$.
Define numbers $a(n)$ by the formal identity
\begin{equation} \label{eq:shimura-l-function-def}
    \mu_j(d) \sum_{n=1}^\infty \frac{a(n)}{n^s} = L(s+\tfrac 12,\chi_d) \sum_{m=1}^\infty \left(\frac{12}{m}\right)\frac{\mu_j(dm^2)}{m^{s-1}},
\end{equation}
where $\chi_d = \legendre d\cdot$, and let
\begin{equation} \label{eq:shimura-lift-def}
    v(z) = S_d(u_j)(z) = 2\sqrt y\sum_{n\neq 0} a(|n|) K_{2ir_j} (2\pi |n|y) e(nx).
\end{equation}
Then $v$ satisfies $T_p v = (\tfrac{12}{p})\lambda_j(p) v$ for all $p\geq 5$, where $T_p$ is the usual Hecke operator in weight $0$ that acts on Fourier expansions as
\begin{equation}
    (T_p v)(z) = 2y^{1/2} \sum_{n\neq 0} \left(a(|pn|) + \chi_0(p)p^{-1}a(|n/p|)\right) K_{2ir_j}(2\pi|n|y) e(nx)
\end{equation}
(see \cite[Corollary~5.2]{AA} for a proof of this; note that the assumption $d>0$ is not used in the proof).
Here $\chi_0$ is the trivial character modulo $6$.
Thus $a(p)=(\tfrac{12}{p})\lambda_j(p)$.
We claim that $v$ is also an eigenform of $T_2$ and $T_3$.
Let $p\in \{2,3\}$.
Then by \eqref{eq:shimura-l-function-def} we have $a(p^\ell) = p^{-\ell/2}$.
If $n=p^\ell n'$ with $p\nmid n'$ then 
\eqref{eq:shimura-l-function-def} yields
\begin{align}
    a(pn)
    &=
    \chi_d(p)^{\ell+1} p^{-(\ell+1)/2}
    \mu_j(d)^{-1}\sum_{km=n'} \chi_d(k) \left(\frac{12}{m}\right) \frac{m}{\sqrt k} \mu_j(dm^2) \\
    &= p^{-(\ell+1)/2} a(n') = p^{-1/2}a(p^\ell)a(n') = p^{-1/2}a(n).
\end{align}
Thus $T_p v = p^{-1/2} v$ for $p\in \{2,3\}$.
It remains to show that $v$ is actually a Maass cusp form of weight $0$.

\begin{proposition} \label{prop:shimura-maass}
Let $d\equiv 1\pmod{24}$ be negative and squarefree.
Then $S_d(u_j) \in \mathcal S_{0}(6,\chi_0,2r_j)$.
Furthermore, $S_d(u_j)$ has eigenvalue $-1$ under the Atkin-Lehner involutions $W_2$ and $W_3$.
\end{proposition}

We will prove Proposition~\ref{prop:shimura-maass} and define $W_2$ and $W_3$ in the next section.

The Shimura lift allows us to estimate $\mu_j(dw^2)$ in terms of $\mu_j(d)$.
Suppose that $(w,6)=1$.
By \eqref{eq:shimura-l-function-def} we have
\begin{equation}
    w\mu_j(dw^2) = \left(\frac{12}{w}\right)\mu_j(d) \sum_{\ell n = w} \mu(\ell) \chi_d(\ell) \ell^{-1/2} a(n) \ll |\mu_j(d)| w^{\ep} \max_{n\mid w} |a(n)|.
\end{equation}
A result of Kim and Sarnak \cite{KS02} gives $|a(n)| \ll |n|^{\theta+\ep}$ where $\theta = \tfrac7{64}$.
It follows that for $n\geq 1$ we have
\begin{equation} \label{eq:n-to-t}
    (24n-1) \sum_{r_j\leq x} \frac{|\mu_j(1-24n)|^2}{\cosh(\pi r_j)} \ll |d| w^{2\theta+\ep} \sum_{r_j\leq x} \frac{|\mu_j(d)|^2}{\cosh \pi r_j}, \qquad \text{ where }1-24n=dw^2.
\end{equation}
For each $j$, choose an orthogonal basis $\mathcal B_{j}$ of the subspace of $\mathcal S_0(6,\chi_0,2r_j)$ of forms with eigenvalue $-1$ under $W_2$ and $W_3$, where each element of $\mathcal B_j$ is an eigenform for all $T_p$.
We arithmetically normalize each $\varphi\in \mathcal B_j$ so that $a(1)=1$.
Then
\begin{equation}
    \sum_{r_j\leq x} \frac{|\mu_j(d)|^2}{\cosh \pi r_j} \leq \sum_{r_j\leq x} \frac{1}{\cosh\pi r_j} \sum_{\varphi\in \mathcal B_{j}} \sum_{S_d(u_j)=\varphi}|\mu_j(d)|^2.
\end{equation}
In the next section we will prove the following formula.%
\footnote{An analogous formula (which can be proved in a similar way) holds when $d$ is positive, but we will not need it here.}
\begin{proposition} \label{prop:sum-walds}
Suppose that $d\equiv 1\pmod{24}$ is negative.
Then for each $\varphi\in \mathcal B_{j}$ we have
\begin{equation} \label{eq:sum-walds}
    \frac{\pi}{3}|d|\sum_{S_d (u_j) = \varphi}|\mu_j(d)|^2 = \langle \varphi,\varphi \rangle^{-1} \left|\Gamma\left(\tfrac 34+ir_j\right)\right|^2 L(\tfrac 12,\varphi\times \chi_d),
\end{equation}
where $L(s,\varphi\times\chi_d)$ is the analytic continuation of the $L$-function
\begin{equation}
    L(s,\varphi\times\chi_d) = \sum_{n=1}^\infty \frac{a(n)\chi_d(n)}{n^s}.
\end{equation}
\end{proposition}
Corollary~0.3 of \cite{HL94} gives the upper bound $\langle \varphi,\varphi \rangle^{-1} \ll (1+r_j)^\ep e^{2\pi r_j}$.
So by Stirling's formula we have
\begin{equation}
    \frac{|d|}{\cosh\pi r_j}\sum_{S_d (u_j) = \varphi} |\mu_j(d)|^2 \ll (1+r_j)^{\frac 12+\ep} L(\tfrac 12,\varphi\times \chi_d).
\end{equation}
Thus
\begin{equation}
    |d| \sum_{r_j\leq x} \frac{|\mu_j(d)|^2}{\cosh \pi r_j} \ll \sum_{r_j\leq x} (1+r_j)^{\frac 12+\ep} \sum_{\varphi\in \mathcal B_j} L(\tfrac 12,\varphi\times \chi_d).
\end{equation}
In Section~\ref{sec:subconvexity} we prove a hybrid subconvexity bound, uniform in both $d$ and $r_j$, for the central $L$-values appearing above.
Theorem~\ref{thm:subconvexity}, together with H\"older's inequality, shows that
\begin{equation} \label{eq:waldspurger-and-subconvexity}
    |d| \sum_{0<r_j\leq x} \frac{|\mu_j(d)|^2}{\cosh \pi r_j}  \ll |d|^{\frac13} x^{\frac 52}(|d|x)^{\ep}.
\end{equation}

We return to the sums $\mathcal M_\ell$ for $\ell=1,2,3$ above.
By \eqref{eq:mean-val-estimate}, \eqref{eq:n-to-t}, \eqref{eq:waldspurger-and-subconvexity}, and  the Cauchy-Schwarz inequality, we find that
\begin{equation}
    \sqrt{24n-1}\sum_{0<r_j\leq x} \frac{|\mu_j(1)\mu_j(1-24n)|}{\cosh(\pi r_j)} \ll x^{3/4} \left(x^{5/4} + \min\left(|dw^2|^{1/4}, x^{5/4}|d|^{1/6}w^{\theta}\right)\right)(nx)^\ep.
\end{equation}
For $\mathcal M_1$ and $\mathcal M_2$ we use that $\min(a,b)\leq b$ to see that
\begin{equation}
    \mathcal M_1 \ll \sqrt{24n-1} \sum_{T=0}^\infty e^{-T/2} \sum_{T<r_j\leq T+1}\frac{|\mu_j(1)\mu_j(1-n)|}{\cosh(\pi r_j)} \ll |d|^{1/6+\ep}w^{\theta+\ep}
\end{equation}
and
\begin{equation}
    \mathcal M_2 \ll x^{-1}|d|^{2/3}w^{1+\theta}(nx)^\ep.
\end{equation}
For $\mathcal M_3(A)$ we use that $\min(a,b)\leq a^{2/5}b^{3/5}$ to get
\begin{equation}
    \mathcal M_3(A) \ll \min(1,A^{-1}x^{1/3})\left(A^{1/2}+|d|^{1/5}w^{(1+3\theta)/5}\right)(nA)^{\ep},
\end{equation}
from which it follows that
\begin{equation}
    \sum_{\ell=0}^\infty \mathcal M_3\left(2^\ell \frac ax\right) \ll
    \left(x^{1/6} + |d|^{1/5}w^{(1+3\theta)/5}
    \right)(nx)^\ep.
\end{equation}
Thus we have
\begin{equation}
    \sum_{x\leq c\leq 2x} \frac{A_c(n)}{c} \ll \left(x^{1/6} + 
    |d|^{1/5}w^{(1+3\theta)/5}
    +x^{-1}|d|^{2/3}w^{1+\theta}
    \right)(nx)^\ep.
\end{equation}
Applying Lehmer's bound \eqref{eq:lehmer-weil-bound} in the range $c\leq |d|^\alpha w^\beta$ we find that
\begin{align}
    \sum_{c\leq x} \frac{A_c(n)}{c} 
    &\ll
    \sum_{c\leq |d|^\alpha w^\beta} \frac{A_c(n)}{c} + \sum_{c > |d|^\alpha w^\beta} \frac{A_c(n)}{c} \\
    &\ll
    \left(|d|^{\alpha/2}w^{\beta/2}+x^{1/6}
    +|d|^{1/5}w^{(1+3\theta)/5}
    +|d|^{2/3-\alpha}w^{1-\beta+\theta}\right)(nx)^\ep.
\end{align}
We choose $\alpha=4/9$ and $\beta=2/3$ to balance the first and third terms and this yields
\begin{equation} \label{eq:sum-acn-main}
    \sum_{c\leq x} \frac{A_c(n)}{c} \ll \left(x^{1/6} + |d|^{2/9}w^{1/3} \right)(nx)^\ep.
\end{equation}
Arguing as in Section~10 of \cite{AA} we conclude that
\begin{equation}
    R(n,\alpha\sqrt n) \ll_\alpha |d|^{-19/36+\ep}w^{-7/6+\ep}.
\end{equation}
This completes the proof of Theorem~\ref{thm:main}.

\section{The Shimura lift and the Waldspurger formula}
\label{sec:shimura}

In this section we prove Propositions~\ref{prop:shimura-maass} and \ref{prop:sum-walds}.
We will follow the arguments given in Sections~8--10 of \cite{DIT-geometric} with modifications following the ideas in \cite{andersen-infinite,andersen-singular}.

We first need Poincar\'e series of weight $0$ for $\Gamma_6 = \Gamma_0(6)/\{\pm I\}$.
Let $\Gamma_\infty = \{\pm \left(\begin{smallmatrix}1&n\\0&1\end{smallmatrix}\right) : n\in \Z\}$.
For any $m\in \Z$ and for $\mathrm{Re}(s)>1$ define
\begin{equation}
    F_{m}(z,s) = \sum_{\gamma\in \Gamma_\infty \backslash \Gamma_6} f_m(\gamma z,s),
\end{equation}
where
\begin{equation}
    f_m(z,s) = 
    \begin{cases}
    y^s & \text{ if } m=0, \\
    \frac{\Gamma(s)}{2\pi\sqrt{|m|}\Gamma(2s)} M_{0,s-\frac12}(4\pi |m|y) e(mx) & \text{ if }m\neq 0.
    \end{cases}
\end{equation}
The function $F_0(z,s)$ is the nonholomorphic Eisenstein series for $\Gamma_6$ (see \cite[Chapter~15]{IK}).
For $m\neq 0$ we have the following analogue of Proposition~3 of \cite{DIT-geometric}.
It can be proved similarly, with only very minor modifications.

\begin{proposition} \label{prop:residue-Fm}
For $r\geq 0$ let $\mathcal B_r$ denote a Hecke-Maass orthogonal basis of the finite-dimensional vector space $\mathcal S_0(6,\chi_0,r)$.
For $\varphi\in \mathcal B_r$, write
\begin{equation}
    \varphi(z) = 2\sqrt y \sum_{m\neq 0} a_\varphi(m) K_{ir}(2\pi |m|y)e(mx).
\end{equation}
Then $F_m(z,s)$ has a meromorphic continuation to $\mathrm{Re}(s)>0$ and
\begin{equation}
    \mathrm{Res}_{s=\frac 12+ir}(2s-1)F_{m}(z,s) = 2\sum_{\varphi\in \mathcal B_r} \langle \varphi,\varphi \rangle^{-1}a_\varphi(m) \varphi(z).
\end{equation}
\end{proposition}

For each $\ell\mid 6$ let $W_\ell$ denote any matrix with determinant $\ell$ of the form $W_\ell = \psmatrix{a\ell}{b}{6c}{d\ell}$ with $a,b,c,d \in \Z$.
Then the map $\varphi(z)\mapsto \varphi(W_\ell z)$ is an Atkin-Lehner involution on $\mathcal S_0(6,\chi_0,r)$ which does not depend on the choice of $a,b,c,d$. 
It is convenient to choose
\begin{equation}
    W_1 = \psmatrix 1001, \quad 
    W_2 = \psmatrix 2{-1}6{-2}, \quad 
    W_3 = \psmatrix 3163, \quad 
    W_6 = \psmatrix 0{-1}60. 
\end{equation}
Since the Atkin-Lehner involutions commute with the action of the Hecke operators, we can choose the orthogonal basis $\mathcal B_r$ in Proposition~\ref{prop:residue-Fm} to have the additional property that each $\varphi$ satisfies 
\[
    \varphi(W_\ell z) = \alpha(\ell) \varphi(z), \qquad \text{ where } \alpha(\ell)=\pm 1.
\]
Additionally, since $\varphi(W_\ell W_{\ell'}z) = \varphi(W_{\ell''}z)$, where $\ell'' = \ell\ell'/(\ell,\ell')^2$, we see that the only valid sign patterns for $(\alpha(1),\alpha(2),\alpha(3),\alpha(6))$ are $(+,+,+,+)$, $(+,-,-,+)$, and $(+,-,+,-)$.
It follows that
\begin{equation} \label{eq:residue-varphi}
    \mathrm{Res}_{s=\frac 12+ir}(2s-1) \sum_{\ell\mid 6} \mu(\ell) F_{m}(W_\ell z,s) = 8\sum_{\varphi\in \mathcal B_r^\ast} \langle \varphi,\varphi \rangle^{-1}a_\varphi(m) \varphi(z),
\end{equation}
where $\mu$ is the M\"obius function and $\mathcal B_r^\ast$ is the subset of $\mathcal B_r$ containing only the $\varphi$ for which $\alpha(\ell)=\mu(\ell)$.
To simplify the notation, let
\begin{equation}
    F_{m}^*(z,s) = \tfrac 14\sum_{\ell\mid 6} \mu(\ell) F_{m}(W_\ell z,s).
\end{equation}
A straightforward modification of the proof of Proposition~5 of \cite{aa-spt} shows that the Fourier expansion of $F_{m}^\ast(z,s)$ is of the form
\begin{equation} \label{eq:Fm-fourier-exp}
    F_{m}^\ast(z,s) = f_{m}(z,s) + c_{m}(s)y^{1-s} + \sqrt{y}\sum_{n\neq 0} c_{m}(n,s) K_{s-\frac 12}(2\pi|n|y)e(nx),
\end{equation}
where $c_m(n,s)\in\C$ and
\begin{equation}
    c_m(s) = 
    \begin{cases}
    \frac{\sqrt{\pi}\Gamma(s-\tfrac 12)}{4\Gamma(s)} \sum_{\ell\mid 6} \frac{\mu(\ell)}{\ell^s} \sum_{\substack{0<c\equiv 0(6/\ell) \\ (c,\ell)=1}} \frac{\varphi(c)}{c^{2s}}
    &\text{ if $m=0$,} \\
    \frac{\pi^{s}m^{s-\frac12}}{4(s-\tfrac 12)\Gamma(s)} \sum_{\ell\mid 6} \frac{\mu(\ell)}{\ell^s} \sum_{\substack{0<c\equiv 0(6/\ell) \\ (c,\ell)=1}} \frac{R_c(m)}{c^{2s}}
    &\text{ if $m\neq0$,}
    \end{cases}
\end{equation}
where $\varphi(c)$ is the Euler totient function, and $R_c(m)$ is the Ramanujan sum with modulus $c$.
By a standard calculation, $c_m(s)$ can be written in terms of $\zeta(2s)$ and $\zeta(2s-1)$.
The exact evaluations are not important for us here, only that $c_m(s)/c_0(s)$ is meromorphic in $\mathrm{Re}(s)>0$ with no pole at $s=\tfrac 12+ir$ if $r\neq 0$.

Let $D\equiv 1\pmod{24}$ be a positive integer and let $\mathcal Q_D$ denote the set of (indefinite) integral binary quadratic forms $Q(x,y)=[a,b,c](x,y)=ax^2+bxy+cy^2$ with discriminant $D=b^2-4ac$ and with $6\mid a$.
Let $\Gamma_6^*$ denote the group generated by $\Gamma_6$ and $\{W_\ell : \ell \mid 6\}$.
A matrix $\gamma=\begin{smallmatrix}A & B \\ C & D\end{smallmatrix} \in \Gamma^*_6$ acts on $Q\in \mathcal Q_D$ via
\begin{equation}
    (\gamma Q)(x,y) = \frac{1}{\det \gamma} Q(Dx-By, -Cx+Ay).
\end{equation}
The set $\Gamma_6\backslash \mathcal Q_D$ contains finitely many equivalence classes and forms a group under Gauss composition.
For $r\in \{1,5,7,11\}$, let $\mathcal Q_D^{(r)}$ be the subset of $\mathcal Q_D$ comprising those $[a,b,c]$ for which $b\equiv r\pmod{12}$.
In Section~3 of \cite{andersen-singular} it is shown that
\begin{equation}
    \mathcal Q_D = \bigcup_{\ell \mid 6} W_\ell \mathcal Q_D^{(1)}.
\end{equation}
When $D$ is a square we have the following explicit description of $\Gamma_6\backslash \mathcal Q_D$, which is a straightforward generalization of Lemma~3 of \cite{andersen-infinite}.

\begin{lemma} \label{lem:square-quad-forms}
Let $\omega\in \{-1,1\}$.
If $D=d^2$ then
    \begin{equation}
		\{[0,\omega|d|,c] : 0\leq c < |d|\}
    \end{equation}
	is a complete set of representatives for $\Gamma_6\backslash \mathcal Q_D^{(r)}$, where $r\equiv \omega|d|\pmod{12}$.
\end{lemma}

For $Q\in \mathcal Q_D$ let $S_Q$ denote the geodesic in $\mathcal H$ connecting the (real) roots of $Q(z,1)=0$.
Explicitly, $S_Q$ is the set of points satisfying
\begin{equation}
    a|z|^2+b\mathrm{Re}z+c = 0.
\end{equation}
When $a\neq 0$, $S_Q$ is a semicircle in $\mathcal H$ which we orient clockwise\footnote{We are following the convention in \cite{DIT-geometric}.} if $a>0$, counterclockwise if $a<0$.
When $a=0$, $S_Q$ is the vertical line $\mathrm{Re}(z) = -c/b$ which we orient downward.
If $D$ is not a square then the group $\Gamma_Q\subseteq \Gamma_6$ of automorphs of $Q$ is infinite cyclic, and if $D$ is a square then $\Gamma_Q$ is trivial.
In either case let $C_Q = \Gamma_Q\backslash S_Q$.
Then $C_Q$ is a closed geodesic (of finite length) on $\Gamma_6\backslash \mathcal H$ if $D$ is not a square, and it is an infinite geodesic otherwise.

Following Section~9 of \cite{DIT-geometric}, we would like to integrate the function $\partial_z F_{m}^\ast (z,s)$ over $C_Q$; when $D$ is not a square there is no issue, but when $D$ is a square the integral does not converge.
In that case we will integrate the function
\begin{equation}
    F_{m,Q}^*(z,s) = \sum_{\ell\mid 6} \mu(\ell)  \sum_{\substack{\gamma\in \Gamma_\infty \backslash \Gamma_6 \\ \gamma W_\ell \mathfrak a_j \neq \infty}} f_m(\gamma W_\ell z,s),
\end{equation}
where $\mathfrak a_1,\mathfrak a_2\in \mathbb P^1(\Q)$ are the endpoints of $C_Q$.
The function $F_{m,Q}^\ast(z,s)$ is studied in Section~3 of \cite{andersen-singular}.
We claim that the endpoints $\mathfrak a_1$ and $\mathfrak a_2$ are related by
\begin{equation}
    \mathfrak a_2 = W_6\gamma \mathfrak a_1
\end{equation}
for some $\gamma\in \Gamma_6$.
This is straightforward to see for the set of forms in Lemma~\ref{lem:square-quad-forms}, and thus holds for all forms in $\mathcal Q_D$.
Since $F_{m,Q}^\ast(z,s)$ is invariant under $z\mapsto W_6 z$, the Fourier expansion of $F_{m,Q}^\ast(z,s)$ is the same at $\mathfrak a_1$ and at $\mathfrak a_2$.
Thus, by \eqref{eq:Fm-fourier-exp}, the Fourier expansion of $F_{m,Q}^\ast(z,s)$ at the cusp $\mathfrak a_j$ is of the form
\begin{equation}
    c_{m,j}(s)y^{1-s} + \sqrt{y}\sum_{n\neq 0} c_{m,j}(n,s) K_{s-\frac 12}(2\pi|n|y)e(nx),
\end{equation}
where $c_{m,1}(s)=c_{m,2}(s)$ and $c_{m,1}(n,s)=c_{m,2}(n,s)$.

For the rest of this section, let $D=dd'$ be a factorization of $D$ into negative discriminants $d,d'\equiv 1\pmod{24}$ where $d$ is squarefree.
The generalized genus character associated to this factorization is
\begin{equation}
    \chi_d(Q) =
    \begin{cases}
    \left(\tfrac dm\right) & \text{ if }\gcd(a,b,c,d)=1 \text{ and } m = Q(x,y)\text{ for some } x,y\in \Z,\\
    0 & \text{ if }\gcd(a,b,c,d)>1.
    \end{cases}
\end{equation}
The following proposition evaluates the cycle integrals in terms of the Kloosterman sums 
\begin{equation}
    S(p_0,q_0,c,\nu_\eta) = \sum_{\left(\begin{smallmatrix} a & b \\ c & d \end{smallmatrix}\right) \in \Gamma_\infty \backslash \Gamma / \Gamma_\infty} \bar\nu_\eta\left(\begin{pmatrix} a & b \\ c & d \end{pmatrix}\right) e\left(\frac{p a + q d}{24c}\right),
\end{equation}
where $p,q\equiv 1\pmod{24}$ and $p_0 = \tfrac{p+23}{24}$, $q_0=\tfrac{q+23}{24}$.
For uniform notation, when $D$ is not a square we define $F_{m,Q}^\ast(z,s) = F_{m}^\ast(z,s)$.

\begin{proposition}
Let $D$ be a positive integer with $D\equiv 1\pmod{24}$ and let $D=dd'$ be a factorization of $D$ into negative integers such that $d\equiv 1\pmod{24}$ is squarefree.
For $m\geq 0$ and $\mathrm{Re}(s)>1$ we have
\begin{multline} \label{eq:cycle-int-prop}
    \sum_{Q\in \Gamma_6\backslash \mathcal Q_D^{(1)}} \chi_d(Q) \int_{C_Q} i \partial_z F_{m,Q}^\ast (z,s) \, dz \\=
    \begin{cases}
        \frac{\Gamma(\frac{s+1}{2})}{\Gamma(\frac s2)}D^{1/4}  \sum_{n\mid m} \left(\frac{12}{m/n}\right) \left(\frac dn\right) n^{-\frac 12} R\left(\frac{m^2}{n^2}d,d',s\right) & \text{ if } m\neq 0, \\
        0 & \text{ if }m=0,
    \end{cases}
\end{multline}
where
\begin{equation}
     R(p,q,s) = 2\sqrt{\pi} \, \sqrt{-i} \sum_{c>0} \frac{S(p_0,q_0,c,\nu_\eta)}{c} J_{s-\frac 12} \left(\frac{\pi \sqrt{|pq|}}{6c}\right), \qquad p_0 = \frac{p+23}{24}.
\end{equation}
\end{proposition}

\begin{proof}
    We closely follow the proof of Proposition~5.1 of \cite{andersen-singular}.
    In the notation of \cite{andersen-singular} we have
    \begin{equation}
        F_{m,Q}^\ast (z,s) = \frac{\Gamma(s)P_{-m,Q}(z,s)}{2^{s+3}\sqrt{|m|}\Gamma(\tfrac{s+1}{2})^2} .
    \end{equation}
    Let $T_m(d,d')$ denote the left-hand side of \eqref{eq:cycle-int-prop} and write
    \begin{equation}
        2i\partial_z F_{m,Q}^\ast (z,s) = \sum_{\ell\mid 6} \mu(\ell)  \sideset{}{'}\sum_{\gamma\in \Gamma_\infty \backslash \Gamma_6} f_{2,m}(\gamma W_\ell z,s) \frac{d(\gamma W_\ell z)}{dz},
    \end{equation}
    where $f_{2,m}(z,s) = \phi_{2,m}(y,s)e(mx)$ and (using (9.2) of \cite{DIT-geometric})
    \begin{equation}
        \phi_{2,m}(y,s) = 
        \begin{cases}
            sy^{s-1} & \text{ if }m=0, \\
            sm^{-1/2}(2\pi y)^{-1} \tfrac{\Gamma(s)}{\Gamma(2s)}M_{1,s-\frac 12}(4\pi my) & \text{ if }m>0.
        \end{cases}
    \end{equation}
    To handle the cases when $D$ is a square and nonsquare together, we have written $\textstyle{\sum_\gamma'}$ to indicate that the terms with $\gamma W_\ell \mathfrak a_j = \infty$ should be excluded when $D$ is a square.
    Then 
    \begin{equation}
        T_m(d,d') = \frac 12 \sum_{Q\in \Gamma_6\backslash \mathcal Q_D^{(1)}} \chi_d(Q) \sum_{\ell\mid 6} \mu(\ell)  \sideset{}{'}\sum_{\gamma\in \Gamma_\infty \backslash \Gamma_6} \int_{C_Q} f_{2,m}(\gamma W_\ell z,s) d(\gamma W_\ell z).
    \end{equation}
    For each $Q\in \mathcal Q_D$ let $\Gamma_Q$ denote the stabilizer of $Q$ (note that when $D$ is a square, $\Gamma_Q$ is trivial).
    Then
    \begin{equation}
        \sideset{}{'}\sum_{\substack{\gamma\in \Gamma_\infty \backslash \Gamma_6}} \int_{C_Q} f_{2,m}(\gamma W_\ell z,s) d(\gamma W_\ell z) = \sideset{}{'}\sum_{\substack{\gamma\in \Gamma_\infty \backslash \Gamma_6 / \Gamma_Q}} \int_{S_Q} f_{2,m}(\gamma W_\ell z,s) d(\gamma W_\ell z).
    \end{equation}
    Making the change of variable $\gamma W_\ell z \mapsto z$ we have
    \begin{equation}
        T_m(d,d') = \frac 12 \sum_{Q\in \Gamma_6\backslash \mathcal Q_D^{(1)}} \chi_d(Q) \sum_{\ell\mid 6} \mu(\ell)
        \sideset{}{'}\sum_{\substack{\gamma\in \Gamma_\infty \backslash \Gamma_6 / \Gamma_Q}} \int_{S_{\gamma W_\ell Q}} f_{2,m}(z,s) dz
    \end{equation}
    
    As explained in \cite{andersen-singular}, the map $(\gamma,\ell,Q)\mapsto \gamma W_\ell Q$ is a bijection from $\Gamma_\infty \backslash \Gamma_6 /\Gamma_Q \times \{1,2,3,6\} \times \Gamma_6\backslash \mathcal Q_D^{(1)}$ to $\Gamma_\infty \backslash \mathcal Q_D$, and
    $\chi_d(\gamma W_\ell Q) = \chi_d(Q)$.
    Furthermore, if $[a,b,c]=W_\ell Q$ for some $Q\in \mathcal Q_D^{(1)}$ then $\mu(\ell) = (\tfrac{12}b)$.
    When $D$ is not a square, each quadratic form $[a,b,c]$ in the set $\Gamma_\infty \backslash \mathcal Q_D$ has $a\neq 0$.
    Now suppose that $D$ is a square and let $Q\in \mathcal Q_D$.
    If $\mathfrak a_Q$ is one of the  roots of $Q(x,y)=0$ then $\gamma W_\ell \mathfrak a_Q = \mathfrak a_{\gamma W_\ell Q}$.
    Thus the bijection $(\gamma,\ell,Q)\mapsto \gamma W_\ell Q$ translates the condition $\gamma W_\ell \mathfrak a_j \neq 0$ to the condition $\mathfrak a_j \neq \infty$.
    The quadratic forms $[a,b,c]\in \Gamma_\infty \backslash \mathcal Q_D$ with one of $\mathfrak a_j=\infty$ are precisely those with $a=0$.
    Thus, for all $D$, we have
    \begin{equation}
        T_m(d,d') = \frac 12 \sum_{\substack{Q\in \Gamma_\infty \backslash \mathcal Q_D \\ Q=[a,b,c],a\neq 0}} \left(\frac{12}{b}\right)\chi_d(Q) \int_{S_Q} f_{2,m}(z,s) \, dz.
    \end{equation}

	Since $\chi_d(-Q) = -\chi_d(Q)$ and the geodesic $S_{-Q}$ is the same set as $S_Q$ but with opposite orientation, we have
	\begin{equation}
		T_m(d,d') = \sum_{\substack{Q\in \Gamma_\infty \backslash \mathcal Q_{D} \\ Q=[a,b,c], a>0}} \left(\frac{12}{b}\right) \chi_d(Q) \int_{S_Q} e(mx) \phi_{2,m}(y,s) \, dz.
	\end{equation}
	Each $S_Q$ with $Q=[a,b,c]$ and $a>0$ can be parametrized by
	\begin{equation}
		z  = \mathrm{Re} z_Q - e^{-i\theta} \mathrm{Im} z_Q, \qquad 0\leq \theta \leq \pi,
	\end{equation}
	where
	\begin{equation}
		z_Q = -\frac{b}{2a} + i\frac{\sqrt{D}}{2a}
	\end{equation}
	is the apex of the geodesic.
	Thus
	\begin{equation} \label{eq:phi-integral}
		\int_{S_Q} e(mx)\phi_{2,m}(y,s) \, dz = e\left(\frac{-mb}{2a}\right) H_m\left(\frac{\sqrt D}{2a}\right),
	\end{equation}
	where
	\begin{equation}
		H_m(t) = it \int_0^\pi e(-mt\cos\theta) \phi_{2,m}(t\sin\theta,s)e^{-i\theta} \, d\theta.
	\end{equation}
	It follows that
	\begin{equation} \label{eq:Tm-Hm-exp-sum}
		T_m(d,d') = \sum_{c=1}^\infty H_m\left(\frac{\sqrt D}{12c}\right) \sum_{\substack{b(12c) \\ b^2\equiv D(24c)}} \left(\frac{12}{b}\right)\chi_d\left(\left[6c,b,\tfrac{b^2-D}{24c}\right]\right) e\left(\frac{-mb}{12c}\right).
	\end{equation}
	
	We claim that the inner sum in \eqref{eq:Tm-Hm-exp-sum} equals zero if $m=0$.
	Indeed, let $r$ be an integer satisfying $(\tfrac{12}r)=-1$ and $r^2\equiv 1\pmod{24c(c,d)}$, and replace $b$ by $rb$ in the sum.
	By \cite[Lemma~3.1, P4]{andersen-singular} the $\chi_d$ factor is invariant under this change of variable. It follows that the sum equals zero.
	
	If $m\neq 0$ then by Lemma~7 of \cite{DIT-geometric} we have
	\begin{equation}
		H_m(t) = \frac{2\sqrt\pi \Gamma(\frac{s+1}{2}) t^{1/2}}{\Gamma(\frac s2)} J_{s-\frac 12} (2\pi|m|t).
	\end{equation}
	We finish the proof by applying Proposition~4.2
	of \cite{andersen-singular},
    which states that
    \begin{multline}
        \sum_{\substack{b(24c) \\ b^2\equiv D(24c)}} \left(\frac{12}{b}\right)\chi_d\left(\left[6c,b,\tfrac{b^2-D}{24c}\right]\right) e\left(\frac{-mb}{12c}\right)
        \\
        = 4\sqrt {-3i} \sum_{n\mid (m,c)} \left(\frac{12}{m/n}\right)\left(\frac{d}{n}\right) \sqrt{\frac nc} \, S\left(\frac{m^2}{24n^2}d+\frac{23}{24},\frac{d'+23}{24},c,\nu_\eta\right).
    \end{multline}
	(note that $\sqrt{-i}\,S(p_0,q_0,c,\nu_\eta) = K(p_0-1,q_0-1;c)$ in the notation of \cite{andersen-singular}).
	In that paper it is assumed that $(m,6)=1$, but that assumption is only used in one step of the proof, namely the key identity just above (4.13) in \cite{andersen-singular}.
	For $(m,6)>1$ that identity reads
	\begin{equation} \label{eq:key-ident}
	   \sum_{j(2u)} e\left(\frac{-d(3j^2+j)/2}{u}+\frac j2\right)
	    \left(e\left(\frac{m(6j+1)}{12u}\right)+e\left(\frac{-m(6j+1)}{12u}\right)\right)=0.
	\end{equation}
	By splitting the sum along $j=2\ell$ and $j=2\ell+1$ we see that \eqref{eq:key-ident} is equivalent to the identity
	\begin{equation}
	    \sum_{h\in \{\pm 1,\pm 7\}} \left(\frac{12}{h}\right) e\left(\frac{-d(\frac{h^2-1}{24})}{c} - \frac{hm}{12c}\right) \sum_{\ell(c)}e\left(\frac{-6d\ell^2+(dh+m)\ell}{c}\right)=0,
	\end{equation}
	which is Lemma~5.5 of \cite{AA}.
	This completes the proof.
\end{proof}

We are ready to prove Proposition~\ref{prop:shimura-maass}.
Let $C(s)$ be a function which is analytic in $\mathrm{Re}(s)>1$. 
Then by Proposition~\ref{eq:cycle-int-prop} we have
\begin{multline} \label{eq:cycle-int-Fm-eis}
    \sum_{Q\in \Gamma_6\backslash \mathcal Q_D^{(1)}} \chi_d(Q) \int_{C_Q} i \partial_z \left(F_{m,Q}^\ast(z,s) - C(s)F_{0,Q}^\ast(z,s)\right) \, dz \\= \frac{\Gamma(\frac{s+1}{2})}{\Gamma(\frac s2)}D^{\frac 14}  \sum_{n\mid m} \left(\frac{12}{m/n}\right) \left(\frac dn\right) n^{-\frac 12} R\left(\frac{m^2}{n^2}d,d',s\right).
\end{multline}
We would like to extend this identity to a neighborhood of $\mathrm{Re}(s)=\tfrac 12$ so that we can apply Proposition~\ref{prop:residue-Fm}.
Note that Corollary~3.6 of \cite{fay} shows that $R(p,q,s)$ has a meromorphic continuation to $\C$ with poles on the line $\mathrm{Re}(s)=\tfrac 12$.
If $D$ is not a square then each $C_Q$ is compact so \eqref{eq:cycle-int-Fm-eis} automatically holds in the region $\mathrm{Re}(s)>\tfrac 12-\ep$.
Suppose that $D$ is a square.
Then the Fourier expansion of $F_{m,Q}^\ast(z,s)-C(s)F_{0,Q}^\ast(z,s)$ at either of the endpoints of $C_Q$ is
\begin{equation}
    \pm\left(c_m(s) - C(s)c_0(s)\right)y^{1-s} + G(z,s),
\end{equation}
where $G(z,s)$ decays exponentially as $y\to\infty$.
Choosing $C(s) = c_m(s)/c_0(s)$, we find that the integrals in \eqref{eq:cycle-int-Fm-eis} are convergent for $\mathrm{Re}(s)>\tfrac 12-\ep$.
Furthermore, we have
\begin{align} \label{eq:residue-obs}
    \Res_{s=\frac 12+ir}\left(F_{m,Q}^\ast(z,s) - C(s)F_{0,Q}^\ast(z,s)\right)
    &= \Res_{s=\frac 12+ir} F_{m,Q}^\ast(z,s) \\
    &= \Res_{s=\frac 12+ir} F_{m}^\ast(z,s)
\end{align}
because $C(s)F_{0,Q}^\ast(z,s)$ is analytic at $s=\tfrac 12+ir$, $r\neq 0$, and because $F_{m,Q}^\ast(z,s)$ differs from $F_m^\ast(z,s)$ by an analytic function.

Corollary~3.6 of \cite{fay} shows that the poles of $(2s-1)R(p,q,s)$ are simple and lie at the points $\tfrac 12\pm ir_j$, where $r_j$ is the spectral parameter of $u_j$ as in Section~\ref{sec:kloo-coeffs}.
From Proposition~\ref{eq:cycle-int-prop}, \eqref{eq:residue-varphi}, and \eqref{eq:residue-obs} we have
\begin{multline} \label{eq:traces-residue}
    2 \sum_{\varphi\in \mathcal B_r^\ast} \langle\varphi,\varphi\rangle^{-1} a_\varphi(m) \sum_{Q\in \Gamma_6 \backslash \mathcal Q_D^{(1)}} \chi_d(Q) \int_{C_Q} i\partial_z\varphi(z) \, dz \\ = D^{\frac 14} \sum_{n\mid m} \left(\frac{12}{m/n}\right) \left(\frac dn\right) n^{-\frac 12} \mathrm{Res}_{s=\frac 12+ir} \frac{(2s-1)\Gamma(\frac{s+1}{2})}{\Gamma(\frac s2)}R\left(\frac{m^2}{n^2}d,d',s\right).
\end{multline}
To compute the residue on the right-hand side of \eqref{eq:traces-residue}, we will follow the argument given in Section~8 of \cite{DIT-geometric}.
We need Poincar\'e series for weight $1/2$ with multiplier system $\nu_\eta$ on $\Gamma_1 = \operatorname{PSL}_2(\Z)$.
These appear as Fourier coefficients of the resolvent kernel $G_{s,\frac 12}(z,z')$ from Theorem~3.1 of \cite{fay}.
For $m\equiv 1\pmod{24}$ and $\mathrm{Re}(s)>1$ define
\begin{equation}
    F_{\frac 12,\frac m{24}}(z,s) = \frac{6\Gamma(s-\frac 14\sgn (m))}{\pi |m|\Gamma(2s)} \sum_{\gamma\in \Gamma_\infty \backslash \Gamma_1} \nu_\eta^{-1}(\gamma) \left(\frac{cz+d}{|cz+d|}\right)^{-\frac 12} M_{\frac 14\sgn m, s-\frac 12}\left(\tfrac{\pi}{6} |m| \mathrm{Im}\gamma z\right) e\left( \tfrac 16 m\mathrm{Re}\gamma z\right).
\end{equation}
This is related to the Poincar\'e series $P_m(z,s)$ from Proposition~8 of \cite{AA-weak} by
\begin{equation}
    F_{\frac 12,\frac m{24}}(z,s) =  \left(\frac{6}{\pi|m|}\right)^{3/4} y^{1/4} \frac{\Gamma(s-\frac 14\sgn (m))}{\Gamma(2s)} P_m(z,s).
\end{equation}
Thus, by that proposition, we have
\begin{multline}
    F_{\frac 12,\frac{m}{24}}(z,s) = \frac{6\Gamma(s-\frac 14\sgn (m))}{\pi |m|\Gamma(2s)} M_{\frac 14\sgn m,s-\frac 12}\left(\tfrac{\pi}{6}|m|y\right) e\left(\tfrac{1}{24}mx\right) \\
    + \frac{6}{\sqrt{\pi}} \sum_{n \equiv 1(24)} |mn|^{-1/2}\frac{\Gamma(s-\frac 14\sgn (m))}{\Gamma(s+\frac 14\sgn (n))} R(m,n,2s-\tfrac 12) W_{\frac 14\sgn(n),s-\frac 12}\left(\tfrac{\pi}{6}|n|y\right) e\left(\tfrac{1}{24}nx\right),
\end{multline}
where $R(p,q,s)$ is defined in Propostion~\ref{eq:cycle-int-prop} for $p,q<0$, and is defined similarly in the other cases (we will not need the precise definition for those cases).
As in Proposition~\ref{prop:residue-Fm}, we have
\begin{equation}
    \Res_{s=\frac 12+\frac{ir}2}(2s-1)F_{\frac 12,\frac m{24}}(z,s) = \sum_{r_j=r/2} \bar{\mu}_j(m) u_j(z),
\end{equation}
where $u_j$, $\mu_j$, and $r_j$ are defined in Section~\ref{sec:kloo-coeffs}.
By equating Fourier coefficients, it follows that
\begin{equation}
    \frac{6}{\sqrt{\pi}}|mn|^{-\frac12}\Res_{s=\frac 12+\frac{ir}2}\frac{(2s-1)\Gamma(s-\frac 14\sgn (m))}{\Gamma(s+\frac 14\sgn (n))} R(m,n,2s-\tfrac 12) = \sum_{r_j=r/2}\bar\mu_j(m)\mu_j(n).
\end{equation}
Thus by \eqref{eq:traces-residue} we have
\begin{multline}
    \sum_{\varphi\in\mathcal B_r^*} \frac{a_\varphi(m)}{\langle \varphi,\varphi \rangle} \sum_{Q\in \Gamma_6\backslash \mathcal Q_D^{(1)}} \chi_d(Q) \int_{C_Q}i\partial_z\varphi(z_Q) \, dz 
    \\= \frac {\sqrt{\pi}}3 D^{\frac34} \sum_{r_j=r/2}  \mu_j(d')
    \sum_{n\mid m} m \left(\frac{12}{m/n}\right) \left(\frac dn\right) n^{-3/2}  \bar\mu_j\left(\tfrac{m^2}{n^2}d\right).
\end{multline}
By \eqref{eq:shimura-l-function-def} the inner sum above equals $\bar \mu_j(d)\bar a(m)$, where $a(m)$ is the $m$-th coefficient of $S_d(u_j)$.
Thus
\begin{equation} \label{eq:shimura=cycle-int}
    \frac {\sqrt{\pi}}3 D^{\frac34} \sum_{r_j=r/2} \mu_j(d') \bar \mu_j(d) S_d(u_j) = \sum_{\varphi\in\mathcal B_r^\ast}  \frac{\varphi}{\langle \varphi,\varphi \rangle} \sum_{Q\in \Gamma_6\backslash \mathcal Q_D^{(1)}} \chi_d(Q) \int_{C_Q}i\partial_z \varphi(z) \, dz,
\end{equation}
where $S_d(u_j)$ is defined in \eqref{eq:shimura-lift-def}.
This identity is valid for all negative $d,d'\equiv 1\pmod{24}$ such that $d$ is squarefree.
Thus, by the same argument given in Section~8 of \cite{biro-cycle} we find that $S_d(u_j) \in \mathcal S_0(6,\chi_0,r)$ and that $S_d(u_j)$ has eigenvalue $-1$ under both $W_2$ and $W_3$.
This completes the proof of Proposition~\ref{prop:shimura-maass}.

We now prove Proposition~\ref{prop:sum-walds}.
We rewrite the left-hand side of \eqref{eq:shimura=cycle-int} as
\begin{equation}
    \frac {\sqrt{\pi}}3 |D|^{\frac34} \sum_{\varphi\in\mathcal B_r^*} \varphi \sum_{S_d(u_j)=\varphi} \mu_j(d') \bar \mu_j(d).
\end{equation}
Then the linear independence of the $\varphi$ yields the formula
\begin{equation} \label{eq:muj-d-muj-dp-int}
    \frac {\sqrt{\pi}}3 |D|^{\frac34} \sum_{S_d(u_j)=\varphi} \mu_j(d') \bar \mu_j(d) 
    = 
    \langle \varphi,\varphi \rangle^{-1} \sum_{Q\in \Gamma_6\backslash \mathcal Q_D^{(1)}} \chi_d(Q) \int_{C_Q}i\partial_z \varphi(z) \, dz.
\end{equation}
    Suppose that $d'=d$.
	Then by Lemma~\ref{lem:square-quad-forms} the quadratic forms $[0,d,c]$ with $0\leq c<|d|$ form a complete set of representatives for $\Gamma\backslash \mathcal Q_{D}^{(1)}$.
	Furthermore,
	\begin{equation}
		\chi_d([0,d,c]) = \left(\frac dc\right).
	\end{equation}
	Following the normalization \eqref{eq:shimura-lift-def}, we write
	\begin{equation}
	    \varphi(z) = 2\sqrt{y}\sum_{n\neq 0}a(n) K_{ir}(2\pi|n|y)e(nx).
	\end{equation}
	A computation involving \cite[\S10.29, (10.30.2), and (10.40.2)]{dlmf} shows that
	\begin{equation}
		\partial_z \left[\sqrt y K_{ir}(2\pi|n|y)e(nx)\right] = \pi i n \sqrt y K_{ir}(2\pi |n|y)e(nx) + g(n,y)e(nx)
	\end{equation}
	for some function $g(n,y)$ which satisfies $g(-n,y)=g(n,y)$ and $g(n,y)\ll |n|^{1/2}e^{-2\pi|n|y}$ as $|n|y\to\infty$ and $g(n,y)\ll_n y^{-1/2}$ as $y\to 0$.
	So if $\mathrm{Re}(s)>1$ we have
	\begin{multline} \label{eq:mellin-transform-g-sum}
		\sum_{Q\in \Gamma\backslash \mathcal Q_{d^2}^{(1)}} \chi_d(Q) \int_{C_Q} i\partial_z\varphi(z) y^{s}\, dz 
		\\
		= -2\pi i\sum_{n\neq 0}n a(n) G(n,d) \int_0^\infty y^{s+\frac 12}K_{ir}(2\pi |n|y) \, dy - 2\sum_{n\neq 0}a(n) G(n,d) \int_0^\infty y^s g(n,y)\,dy,
	\end{multline}
	where $G(n,d)$ is the Gauss sum
	\begin{equation}
	    G(n,d) = \sum_{c\bmod |d|} \left(\frac{d}{c}\right) e\left(\frac{-nc}{d}\right) = i\left(\frac{d}{n}\right) \sqrt{|d|}.
	\end{equation}
	Since $a(-n)G(-n,d)g(-n,y) = -a(n)G(n,d)g(n,y)$, the second sum on the right-hand side of \eqref{eq:mellin-transform-g-sum} vanishes.
	By \cite[(10.43.19)]{dlmf} we have
	\begin{equation} \label{eq:K-integral}
		\int_0^\infty y^{s+\frac 12}K_{ir}(2\pi|n|y) \, dy = \tfrac 14 (\pi|n|)^{-s-\frac 32} \Gamma(\tfrac s2+\tfrac{ir}{2} + \tfrac 34) \Gamma(\tfrac s2-\tfrac{ir}{2} + \tfrac 34).
	\end{equation}
	It follows that
	\begin{equation}
		\sum_{Q\in \Gamma\backslash \mathcal Q_D} \chi_d(Q) \int_{C_Q} i\partial_z\varphi(z) y^{s}\, dz = \pi^{-s-\frac 12} \sqrt{|d|} \, \Gamma(\tfrac s2+\tfrac {ir}2+\tfrac 34)\Gamma(\tfrac s2-\tfrac {ir}2+\tfrac 34) L(s+\tfrac 12,\varphi\times\chi_d).
	\end{equation}
Setting $s=0$ and using \eqref{eq:muj-d-muj-dp-int}, we obtain \eqref{eq:sum-walds}.

\section{Cubic Moment and Subconvexity} \label{sec:subconvexity}

	Let $N \geq 1$ be a fixed integer. Denote by $\mathcal S_0(N)$ the space of weight $0$ Maass cusp forms of level $\Gamma_0(N)$ with a basis $\Bas_0(N)$ of eigenforms for Hecke operators. For a Maass form $\varphi \in \mathcal S_0(N)$, denote by $r = r(\varphi)$ the spectral parameter.
	The purpose of this section is to establish the following bound.
	
\begin{theorem} \label{thm:subconvexity}
	Let $T \gg 1$, and $q$ be a fundamental discriminant. Let $\chi_q$ be the corresponding quadratic Dirichlet character of modulus $\norm[q]$. Then we have for any $\ep > 0$
	$$ \sum_{\substack{\varphi \in \Bas_0(N) \\ T \leq \norm[r(\varphi)] < T+1}} \frac{L \left( \frac{1}{2},\varphi \times \chi_q \right)^3}{L(1,\varphi,\mathrm{Ad})} \ll_{\ep} (T \norm[q] N)^{\ep} T \lcm(\norm[q],N). $$
\label{CubMomentBd}
\end{theorem}
\begin{corollary}
	For the above $\varphi \in \Bas_0(N)$, we have the bound
	$$ L \left( \tfrac{1}{2}, \varphi \times \chi_q \right) \ll_{\ep} (|r_\varphi| \norm[q] N)^{\ep} \left( |r_\varphi| \mathrm{lcm}(\norm[q],N) \right)^{\frac{1}{3}}. $$
\end{corollary}

	Our main tool is the version of Motohashi's formula in \cite{Wu11}, which we recall as follows. For a Hecke eigenform $\varphi$, we do not distinguish it from the irreducible representation $\pi$ generated by it. Hence $\varphi_v$ will mean $\pi_v$ if $\pi \simeq \otimes_v' \pi_v$. Let $S$ be a finite set of primes $p$. Let $\Psi_v \in \Sch(\Mat_2(\Q_v))$ be Schwartz functions at $v \in \{ \infty \} \cup S$. For any Hecke-Maass form $\varphi$ which is unramified at all $p \notin S$, we introduce for each $v \in S \cup \{ \infty \}$
	$$ M_{3,v}(\Psi_v \mid \varphi_v) = \sum_{e_1,e_2 \in \Bas(\varphi_v)} \Zeta_v \left( \tfrac{1}{2}, \Psi_v, \beta(e_2,e_1^{\vee}) \right) \Zeta \left( \tfrac{1}{2}, W_{e_1} \right) \Zeta \left( \tfrac{1}{2}, W_{e_2^{\vee}} \right), $$
	where 
\begin{itemize}	
	\item $\Bas(\varphi_v)$ is an orthogonal basis of $\varphi_v$; 
	\item for $e \in \Bas(\varphi_v)$, $e^{\vee}$ is the dual vector in the dual basis, where the implicit inner product is defined in the Kirillov model;
	\item $W_*$ is the Kirillov function of $*$ with respect to the standard additive character $\psi_v$ \`a la Tate;
	\item $\beta(e_2,e_1^{\vee})$ is the matrix coefficient related to $e_2$ and $e_1^{\vee}$;
	\item $\Zeta_v(s,\Psi_v,\beta)$ is the Godement-Jacquet zeta integral
	\[
	\Zeta_v(s,\Psi_v,\beta) = \int_{\GL_2(\F_v)} \Psi_v(g) \beta(g) \norm[\det g]_v^{s+\frac{1}{2}} dg
	\]
	and $\Zeta_v(s,W)$ is the standard local zeta integral
	\[
	    \Zeta_v(s,W)=\int_{\F_v^{\times}} W(t) \norm[t]_v^{s-\frac{1}{2}} d^{\times}t.
	\]
\end{itemize}
	With these local terms, we define
	$$ M_3(\Psi \mid \varphi) := \frac{3}{\pi} \frac{L \left( \frac{1}{2},\varphi \right)^3}{L(1, \varphi, \mathrm{Ad})} \cdot M_{3,\infty}(\Psi_{\infty} \mid \varphi_{\infty}) \prod_{p \in S} M_{3,p}(\Psi_p \mid \varphi_p) \frac{L(1, \varphi_p \times \bar{\varphi}_p)}{L\left( \frac{1}{2},\varphi_p \right)^3}. $$
	There is a counterpart for the Eisenstein series $M_3(\Psi \mid \chi,s)$. Basically its corresponding local terms $M_{3,v}(\Psi_v \mid \chi_v,s)$ are the same as $M_{3,v}(\Psi_v \mid \varphi_v)$ if $\varphi_v \simeq \pi(\chi_v\norm_v^s, \chi_v^{-1}\norm_v^{-s})$, except that we change the inner product structure to be defined in the induced model. Namely,
\begin{multline} 
	M_3(\Psi \mid \chi,s) = \frac{L\left( \frac{1}{2}+s,\chi \right)^3 L\left( \frac{1}{2}-s,\chi^{-1} \right)^3}{L(1+2s,\chi^2) L(1-2s,\chi^{-1})^2} \cdot M_3(\Psi_{\infty} \mid \chi_{\infty},s) \\ \times \prod_{p \in S} M_{3,p}(\Psi_p \mid \chi_p, s) \frac{L_p(1+2s,\chi_p^2) L_p(1-2s,\chi_p^{-2})}{L_p\left( \frac{1}{2}+s, \chi_p \right)^3 L_p\left( \frac{1}{2}-s, \chi_p^{-1} \right)^3}. 
\end{multline}
	For any Dirichlet character $\chi$ which is unramified at all $p \notin S$, we introduce for each $v \in S \cup \{ \infty \}$
	$$ M_{4,v}(\Psi_v \mid \chi_v) = \zeta_v(1)^4 \int_{\F_v^4} \chi_v \left( \frac{x_1 x_4}{x_2 x_3} \right) \frac{\sideset{}{_i} \prod dx_i}{\norm[x_1x_2x_3x_4]^{\frac{1}{2}}} $$
	and the corresponding global distribution
	$$ M_4(\Psi \mid \chi) := L(\tfrac 12,\chi)^2 L(\tfrac 12,\chi^{-1})^2 \cdot M_{4,\infty}(\Psi_{\infty} \mid \chi_{\infty}) \prod_{p \in S} M_{4,p}(\Psi_p \mid \chi_p) \frac{1}{L(\tfrac 12,\chi_p)^2 L(\tfrac 12,\chi_p^{-1})^2}. $$
	Write
	$$ M_3(\Psi) = \sum_\varphi M_3(\Psi \mid \varphi) + \sum_{\chi} \int_{-\infty}^{\infty} M_3(\Psi \mid \chi, ir) \frac{dr}{4\pi}, $$
	$$ M_4(\Psi) = \sum_{\chi} \int_{-\infty}^{\infty} M_4(\Psi \mid \chi \norm_{\A}^{ir}) \frac{dr}{2\pi}. $$
	Then the formula is
\begin{equation} 
	M_3(\Psi) + DS(\Psi) = M_4(\Psi) + DG(\Psi), 
\label{MF}
\end{equation}
	where the degenerate terms $DS(\cdot)$ and $DG(\cdot)$ are given by
\begin{equation} 
	DG(\Psi) = \Res_{s=\frac{1}{2}} M_4(\Psi \mid \norm_{\A}^s) - \Res_{s=-\frac{1}{2}} M_4(\Psi \mid \norm_{\A}^s), 
\label{DGDef}
\end{equation}
\begin{equation}
	DS(\Psi) = \Res_{s=\frac{1}{2}} M_3(\Psi \mid \id, s). 
\label{DSDef}
\end{equation}

	In order to apply the formula \eqref{MF} to our problem, let $S$ be the set of primes $p \mid qN$. Let $\chi_0$ be the Hecke character corresponding to $\chi_q$. We first specify $\Psi_p$ for $p \in S$ and give the relevant estimations of $M_{3,p}$ and $M_{4,p}$. The choice of $\Psi_p$ is a simple variant of those made in \cite[\S 1.4]{BFW21}. For integers $n \geq 0$, let
	$$ \gp{K}_0[p^n] := \left\{ \begin{pmatrix} a & b \\ c & d \end{pmatrix} \in \GL_2(\Z_p) \ \middle| \ c \in p^n \Z_p \right\}. $$
	For each $p \in S$, write $n_p \geq 1$ such that $p^{n_p} \parallel \mathrm{lcm}(q,N)$, and introduce
	$$ \phi_0(g) := \mathbbm{1}_{\gp{K}_0[\vp^{n_p}]}(g) \chi_{0,p}(\det g). $$
	Let integer $l_p \geq 0$ be such that $p^{l_p} \parallel q$ ($l_p \leq 1$ if $p \neq 2$). Then we take
	$$ \Psi_p = \left\{ \begin{matrix} \rpL_{n(p^{-l_p})} \rpR_{n(p^{-l_p})} \phi_0 & \text{if } p \mid q, \\ \phi_0 & \text{if } p \nmid q, \end{matrix} \right.  $$
	where $\rpL_{g_1} \rpR_{g_2} \Psi(x) := \Psi(g_1^{-1} x g_2)$. This choice differs from \cite[\S 1.4 \& \S 4]{BFW21} in that $l_p$ may be strictly smaller than $n_p$. We first treat the weight functions $M_{3,p}(\cdot)$. The argument is almost the same as \cite[Lemma 4.1]{BFW21}, hence we give a proof with minimal amount of details.
\begin{lemma}
	(1) The local weight $M_{3,p}(\Psi_p \mid \varphi_p) \neq 0$ only if $p^{n_p+1}$ does not divide the level of the form $\varphi \otimes \chi_0$. In this case, we have uniformly in $p$
	$$ M_{3,p}(\Psi_p \mid \varphi_p) \gg p^{-n_p-l_p}, $$
	where the implicit constant depends only on a constant towards the Ramanujan-Petersson conjecture.
	
\noindent (2) If $p \mid q$ but $p \nmid N$, then we have $M_{3,p}(\Psi_p \mid \id, s) = 0$.

\noindent (3) If $p \mid N$, then we have $M_{3,p}(\Psi_p \mid \id, s) \neq 0$ only if $n_p \geq 2l_p$. Under this condition, we have
	$$ \extnorm{ \left. \frac{\partial^n}{\partial s^n} \right|_{s=\frac{1}{2}} M_{3,p}(\Psi_p \mid \id, s) \zeta_p(1+2s) \zeta_p(1-2s) } \ll_n p^{1-n_p-l_p} \log^n p. $$
\label{WtLocBd}
\end{lemma}
\begin{proof}
	(1) Let $\pi$ be the global representation of $\GL_2(\A)$ corresponding to the newform $\varphi$. Let $(\pi_p, V_p)$ be the local component of $\pi$ at $p$. As in the proof of \cite[Lemma 4.1]{BFW21}, we have $\pi_p(\phi_0) \neq 0$ only if $V_p$ contains a non zero vector $e$ satisfying
\begin{equation} 
	\pi_p \begin{pmatrix} a & b \\ c & d \end{pmatrix}.e = \chi_{0,p}(ad) e, \quad \forall \begin{pmatrix} a & b \\ c & d \end{pmatrix} \in \gp{K}_0[p^{n_p}]. 
\end{equation}
	This is equivalent to $\cond(\pi_p \otimes \chi_{0,p}) \leq n_p$. Let $\Bas_0$ be an orthogonal basis of the subspace of $V_p \otimes \chi_{0,p}$ of vectors satisfying 
	$$ (\pi_p \otimes \chi_{0,p}) \begin{pmatrix} a & b \\ c & d \end{pmatrix}.e = e, \quad \forall \begin{pmatrix} a & b \\ c & d \end{pmatrix} \in \gp{K}_0[p^{n_p}], $$
	which contains the newvector $e_0$. This basis determines a dual basis $\Bas_0^{\vee}$ in the dual representation $(\pi_p^{\vee} \otimes \chi_{0,p}, V_p^{\vee} \otimes \chi_{0,p})$, such that $e^{\vee}$ is the dual element of $e$ so that
	$$ (e^{\vee},e)=1, \quad (e^{\vee},e') = 0, \forall e \neq e' \in \Bas_0. $$
	For $e$ or $e^{\vee}$, let $W_e$ or $W_{e^{\vee}}$ be the corresponding Kirillov function of $e$ or $e^{\vee}$ with respect to the standard additive character $\psi_p$ resp. $\psi_p^{-1}$. The above local pairing is defined in the Kirillov model. Then we have
	$$ M_{3,p}(\Psi_p \mid \varphi_p) = \Vol(\gp{K}_0[p^{n_p}]) \sum_{e \in \Bas_0} \Zeta \left( \tfrac{1}{2}, \chi_{0,p}, n(p^{-l_p}).W_e \right) \Zeta \left( \tfrac{1}{2}, \chi_{0,p}^{-1}, n(p^{-l_p}).W_{e^{\vee}} \right), $$
	where $\Zeta(\cdot)$ is the local Rankin-Selberg zeta functional for $\GL_2 \times \GL_1$. Each summand on the right hand side is non-negative. We drop all but the term corresponding to $e=e_0$. Writing $W_0=W_{e_0}$, such that $\mathrm{supp}(W_0) \subset \Z_p$ and $W_0(1)=1$. If $l_p > 0$, then we have by \cite[Proposition 4.6]{Wu14}
	$$ \extnorm{ \Zeta \left( \tfrac{1}{2}, \chi_{0,p}, n(p^{-l_p}).W_0 \right) } = p^{-\frac{l_p}{2}}(1-p^{-1})^{-1}, \quad \Norm[W_0]^2 \ll_{\theta} 1 $$
	where $\theta$ is any constant towards the Ramanujan-Petersson conjecture. While if $l_p = 0$, then we have
	$$ \extnorm{ \Zeta \left( \tfrac{1}{2}, \chi_{0,p}, W_0 \right) } = \extnorm{ L\left( \tfrac{1}{2}, \pi_p \otimes \chi_{0,p} \right) } \gg_{\theta} 1. $$
	The desired bound follows readily.

\noindent (2) This follows from $\cond(\pi(\norm_p^s, \norm_p^{-s}) \otimes \chi_{0,p}^{-1}) = 2 \cond(\chi_{0,p}) > 0$.

\noindent (3) For simplicity, we only treat the case $l \geq 1$. For $\pi_p = \pi(\norm_p^s, \norm_p^{-s})$, The above $\Bas_0$ can be chosen as $\{ e_k : 0 \leq k \leq n_p \}$, where $e_k$ can be written as linear combination of $a(p^{-l})$ translations of $e_0$ as in \cite[Lemma 2.18 (i)]{Wu2}. Let $W_k$ resp. $W_k^{\vee}$ be the Kirillov function of $e_k$ (resp. $e_k^{\vee}$) with respect to $\psi_p$ (resp. $\psi_p^{-1}$). We deduce
	$$ W_0(\id) W_0^{\vee}(\id) = (1-p^{-1-2s})(1-p^{-1+2s}), \quad W_1(\id) W_1^{\vee}(\id) = (p^s+p^{-s})^2, $$
	$$ W_2(\id) W_2^{\vee}(\id) = p^{-2} \frac{1+p^{-1}}{1-p^{-1}}, \quad W_k(\id)W_k^{\vee}(\id)=0, k \geq 3. $$
	The above formula for $M_{3,p}$ can be rewritten via \cite[Proposition 4.6]{Wu14} as
	$$ M_{3,p}(\Psi_p \mid \id,s) = \Vol(\gp{K}_0[p^{n_p}]) p^{-l_p}(1-p^{-1})^2 \sum_{k=0}^{n_p} W_k(\id) W_k^{\vee}(\id). $$
	The desired bound follows readily.
\end{proof}

	The dual weight $M_{4,p}(\Psi_p \mid \chi_p)$ is given by the formula
\begin{align}
	\frac{M_{4,p}(\Psi_p \mid \chi_p)}{\zeta_p(1)^4} &= \int_{\Q_p^4} \phi_0 \begin{pmatrix} x_1 & x_2 \\ x_3 & x_4 \end{pmatrix} \chi_p \left( \frac{(x_1+p^{-l_p}x_3) (x_4 - p^{-l_p}x_3)}{(x_2 - p^{-l_p}(x_1-x_4) - p^{-2l_p}x_3) x_3} \right) \nonumber \\
	&\quad \times\frac{\prod dx_i}{\extnorm{ (x_1+p^{-l_p}x_3) (x_4 - p^{-l_p}x_3) (x_2 - p^{-l_p}(x_1-x_4) - p^{-2l_p}x_3) x_3 }^{\frac{1}{2}}}. \label{DWtFNA}
\end{align}
	In the rest of this paragraph, we omit the subscript $p$ for simplicity of notation. Considering the change of variables
	$$ x_3 \mapsto x_3(1+\delta), \quad x_1 \mapsto x_1 - p^{-l} \delta x_3, \quad x_2 \mapsto x_2 - p^{-2l} \delta x_3, \quad x_4 \mapsto x_4 + p^{-l} \delta x_3 $$
	for any $\delta \in p^l \Z_p$, we get $M_4(\Psi \mid \chi) = \chi(1+\delta)^{-1} M_4(\Psi \mid \chi)$. Thus 
	$$ M_4(\Psi \mid \chi) \neq 0 \quad \Rightarrow \quad m:=\cond(\chi) \leq l. $$
	Note that $\phi_0$ is given in the coordinates of the Bruhat decomposition as
	$$ \phi_0 \left( \begin{pmatrix} z & \\ & z \end{pmatrix} \begin{pmatrix} 1 & \\ x & 1 \end{pmatrix} \begin{pmatrix} 1 & y \\ & 1 \end{pmatrix} \begin{pmatrix} u & \\ & 1 \end{pmatrix} \right) = \mathbbm{1}_{(\Z_p^{\times})^2}(z,u) \mathbbm{1}_{\Z_p^2}(p^{-n}x,y) \chi_0(u)^{-1}. $$
	Making the change of variables
	$$ \begin{pmatrix} x_1 & x_2 \\ x_3 & x_4 \end{pmatrix} = \begin{pmatrix} z & \\ & z \end{pmatrix} \begin{pmatrix} 1 & \\ x & 1 \end{pmatrix} \begin{pmatrix} 1 & y \\ & 1 \end{pmatrix} \begin{pmatrix} u & \\ & 1 \end{pmatrix} = \begin{pmatrix} zu & zy \\ zux & z(1+xy) \end{pmatrix} $$
	whose Jacobian is equal to $\norm[z^3u]$, then $x \mapsto p^nx$, we get
\begin{align*}
	\frac{M_4(\Psi \mid \chi)}{\zeta_p(1)^4} &= p^{-n} \int_{(\Z_p^{\times})^2} \int_{\Z_p^2} \chi \left( \frac{(1+p^{n-l}x) (1-p^{n-l}x(u-p^ly))}{p^{n-l}x (1-(1+p^{n-l}x)(u-p^l y))} \right) \chi_0(u)^{-1}  \\
	&\quad \times\frac{\norm[z] dz du dx dy}{\extnorm{ p^{n-l}x (1+p^{n-l}x) (1-p^{n-l}x(u-p^ly)) (1-(1+p^{n-l}x)(u-p^l y)) }^{\frac{1}{2}}}.
\end{align*}
	(1) If $n=l \geq 1$, then the above integral is exactly the one studied in \cite[\S 4]{BFW21}, and we get uniformly
\begin{equation} 
	\extnorm{ M_4(\Psi \mid \chi) } \ll p^{-2l}. 
\label{DualWtBdq}
\end{equation}
	(2) If $n > l \geq 1$, then by the change of variable $u \mapsto u+p^ly$ we can further simplify
	$$ \frac{M_4(\Psi \mid \chi)}{\zeta_p(1)^4} = \frac{p^{-n}}{\zeta_p(1)} \int_{\Z_p^{\times}} \int_{\Z_p} \chi \left( \frac{ (1+p^{n-l}x) (1-p^{n-l} xu) }{ p^{n-l}x (1-(1+p^{n-l}x)u) } \right) \chi_0(u)^{-1} \frac{dudx}{\extnorm{ p^{n-l}x  (1-(1+p^{n-l}x)u) }^{\frac{1}{2}}}. $$
	Applying the consecutive changes of variables $u \mapsto u^{-1}$ and $u \mapsto 1+p^{n-l}x+y$ gives
\begin{align*}
	\frac{M_4(\Psi \mid \chi)}{\zeta_p(1)^4} &= \frac{p^{-n}}{\zeta_p(1)} \int_{\Z_p^2} \mathbbm{1}_{\Z_p^{\times}}(1+y) \chi \left( \frac{ (1+p^{n-l}x) (1+y) }{p^{n-l}x y} \right) \chi_0(1+p^{n-l}x+y) \frac{dxdy}{\extnorm{ p^{n-l}x y }^{\frac{1}{2}}} \\
	&= \frac{p^{-l}}{\zeta_p(1)} \sum_{k \geq n-l, \ell \geq 0} \int_{p^k \Z_p^{\times} \times p^{\ell} \Z_p^{\times}} \mathbbm{1}_{\Z_p^{\times}}(1+y) \chi \left( \frac{ (1+x) (1+y) }{x y} \right) \chi_0(1+x+y) \frac{dxdy}{\norm[xy]^{\frac{1}{2}}} \\
	&=: \frac{p^{-l}}{\zeta_p(1)} \sum_{k \geq n-l, \ell \geq 0} M_4^{k,\ell}(\chi_0 \mid \chi),
\end{align*}
	where the integrals $M_4^{k,\ell}$ are the same as those defined in \cite[(4.4)]{BFW21}. In other words, the above integral is simply a partial one of the integral studied in the previous case, whose proof actually goes by bounding each summand $M_4^{k,\ell}$. We conclude the bound \eqref{DualWtBdq} in this case, too. \\
	(3) If $l=0$ (and $n \geq 1$), we take the definition of the dual weight formula \eqref{DWtFNA} to see
	$$ \frac{M_{4}(\Psi \mid \chi)}{\zeta_p(1)^4} = \int_{\Q_p^4} \mathbbm{1}_{\gp{K}_0[p^n]} \begin{pmatrix} x_1 & x_2 \\ x_3 & x_4 \end{pmatrix} \chi \left( \frac{x_1x_4}{x_2x_3} \right) \frac{\prod dx_i}{\norm[x_1x_2x_3x_4]^{\frac{1}{2}}}, $$
	which is non-vanishing only if $\cond(\chi)=0$, in which case
\begin{equation} 
	M_4(\Psi \mid \chi) = \int_{\Z_p} \chi(x_2)^{-1} \norm[x_2]^{\frac{1}{2}} d^{\times}x_2 \int_{p^n \Z_p} \chi(x_3)^{-1} \norm[x_3]^{\frac{1}{2}} d^{\times}x_3 \asymp p^{-\frac{n}{2}}. 
\label{DualWtBdN}
\end{equation}

\begin{lemma}
	(1) The local weight $M_{4,p}(\Psi_p \mid \chi_p)$ is non vanishing only if $\cond(\chi_p) \leq l_p$. Under this condition, we have uniformly for unitary $\chi_p$
	$$ \extnorm{ M_{4,p}(\Psi_p \mid \chi_p) } \ll \left\{ \begin{matrix} p^{-2l_p} & \text{if } l_p > 0 \\ p^{-\frac{n_p}{2}} & \text{if } l_p=0 \end{matrix} \right. . $$
	
\noindent (2) For any integer $n \geq 0$, we have
	$$ \extnorm{ \left. \frac{\partial^n}{\partial s^n} \right|_{s=\pm \frac{1}{2}} M_{4,p}(\Psi_p \mid \norm_p^s) \,\zeta_p \left( \tfrac{1}{2}+s \right)^{-2} \zeta_p \left( \tfrac{1}{2}-s \right)^{-2} } \ll_n p^{-l_p} \log^n p. $$
\label{DualWtLocBd}
\end{lemma}
\begin{proof}
	(1) This is just a summary of the above discussion.
	
\noindent (2) We simply replace $\chi$ by $\norm_p^s$ in the above discussion to find precise formula of the term on the left hand side. The details can be found in \cite[\S 5]{BFW21}.
\end{proof}

	At the infinite place, if $\varphi_{\infty}$ has spectral parameter $r$, then we can choose for $T \geq 1$ and $\Delta = T^{\ep}$
	$$ M_{3,\infty}(\Psi_{\infty} \mid \varphi_{\infty}) = \sqrt{\pi} \frac{\cosh(\pi r)}{2 \Delta} \left\{ \exp \left( - \frac{(r-T)^2}{2\Delta^2} - \frac{\pi}{2} r \right) + \exp \left( - \frac{(r+T)^2}{2\Delta^2} + \frac{\pi}{2} r \right) \right\}^2. $$
	This is a positive weight function, which approximates the characteristic function of the interval $[T-\Delta, T+\Delta]$. The dual weight $M_{4,\infty}(\Psi_{\infty} \mid \chi_{\infty})$ is studied in \cite[\S 3.3]{BFW21}, which is bounded as
	$$ \extnorm{ M_{4,\infty}(\Psi_{\infty} \mid \chi_{\infty}) } \ll \left\{ \begin{matrix} 1 & \text{for all } \chi_{\infty}, \\ (1+\norm[T])^{-A} & \text{for any } A \text{ and } \chi_{\infty}(t)=t^{ix} \text{ with } t>0, x \geq (1+\norm[T]) \log(1+\norm[T]). \end{matrix} \right.  $$
\begin{lemma}
	The above local weight $M_{3,\infty}(\Psi_{\infty} \mid \varphi_{\infty})$ satisfies:
\begin{itemize}
	\item[(1)] It is non negative for all unitary $\varphi_{\infty}$.
	\item[(2)] For $\varphi_{\infty}$ with spectral parameter $r$ such that $T-1 \leq \norm[r] \leq T+1$, we have $M_{3,\infty}(\Psi_{\infty} \mid \varphi_{\infty}) \gg_{\ep} T^{-\ep}$.
\end{itemize}
	The dual weight satisfies
	$$ \extnorm{ M_{4,\infty}(\Psi_{\infty} \mid \chi_{\infty}) } \ll \left\{ \begin{matrix} 1 & \text{for all unitary } \chi_{\infty}, \\ \norm[x]^{-A} & \text{for any } A>1, \chi_{\infty}(t)=t^{ix} \text{ with } t>0, x \geq \norm[T] \log^2 \norm[T]. \end{matrix} \right.  $$
\end{lemma}
\begin{lemma}
	The corresponding weight $M_{3,\infty}(\Psi_{\infty} \mid \id, s)$ vanishes at $s=1/2$ to order one, and satisfies for any integer $n \geq 0$ and constant $A > 1$
	$$ \extnorm{ \left. \frac{\partial^n}{\partial s^n} M_{3,\infty}(\Psi_{\infty} \mid \id, s) \right|_{s=\frac{1}{2}} } \ll_{n,A} T^{-A}. $$
	The corresponding dual weight $M_{4,\infty}(\Psi_{\infty} \mid \chi_{\infty} \norm^s)$ has a singularity at $s=\pm 1/2$ of order $\leq 2$, and satisfies for any integer $n \geq 0$
	$$ \extnorm{ \left. \frac{\partial^n}{\partial s^n} \left( s \mp \tfrac{1}{2} \right)^2 M_{4,\infty}(\Psi_{\infty} \mid \id, s) \right|_{s=\pm \frac{1}{2}} } \ll_{\ep, n} T^{1+\ep}. $$
\label{DABd}
\end{lemma}
\noindent For a proof of the above two lemmas, see \cite[\S 5]{BFW21}.

	Inserting Lemma \ref{WtLocBd} (2), Lemma \ref{DualWtLocBd} (2), Lemma \ref{DABd} into the formulas \eqref{DGDef} and \eqref{DSDef}, we deduce
	$$ \extnorm{DG(\Psi)} + \extnorm{DS(\Psi)} \ll_{\ep} T^{1+\ep} q^{-1+\ep}. $$
	Introduce the decomposition
	$$ N=N_0 N_1, $$
	so that $p \mid \mathrm{gcd}(q,N) \Leftrightarrow p \mid q_0 \Leftrightarrow p \mid N_0$, and $\mathrm{gcd}(N_0,N_1)=1$. Inserting all the above local bounds in the Motohashi's formula \eqref{MF}, we get
\begin{align*} 
	\sum_{\substack{\varphi \in \Bas_0(N) \\ r(\varphi) \in [T-\Delta, T+\Delta]}} \frac{L \left( \frac{1}{2},\varphi \times \chi_q \right)^3}{L(1, \varphi, \mathrm{Ad})} &\ll_{\ep} (T\norm[q]N)^{\ep} \mathrm{lcm}(\norm[q],N) \left( T + \frac{N_1^{\frac{1}{2}}}{\norm[q] N_1} \sum_{\chi} \int_{\norm[r] \leq T \log^2 T} \extnorm{L \left( \tfrac{1}{2}+ir,\chi \right)}^4 dr \right), 
\end{align*}
	where the sum over $\chi$ are those Dirichlet characters of conductor dividing $q$. A spectral large sieve inequality shows that the fourth moment is bounded as
	$$ \sum_{\chi} \int_{\norm[r] \leq T \log^2 T} \extnorm{L \left( \tfrac{1}{2}+ir,\chi \right)}^4 dr \ll_{\ep} (T\norm[q])^{1+\ep}. $$
	We deduce the desired bound in Theorem \ref{CubMomentBd}.

\section*{Acknowledgement}

Nickolas Andersen is supported by the Simons Foundation, award number 854098.
Han Wu is supported by the Leverhulme Trust Research Project Grant RPG-2018-401.

\bibliographystyle{acm}
	
\bibliography{mathbib}

\address{\quad \\ Nickolas Andersen \\ Brigham Young University \\ Department of Mathematics \\ Provo, UT 84604, USA \\ {\tt nick@math.byu.edu}}

\address{\quad \\ Han Wu \\ School of Mathematical Sciences \\ Queen Mary University of London \\ Mile End Road \\ E1 4NS, London, UK \\ {\tt wuhan1121@yahoo.com}}
	
\end{document}